\documentclass[12pt,twoside,final,psamsfonts]{amsart}

\usepackage[psamsfonts]{amssymb}
\usepackage{times,a4wide}

\usepackage{color}		
\usepackage{epsfig}

\theoremstyle{plain}
\newtheorem*{theorem*}{Theorem}

\newtheorem{theorem}{Theorem}[section]
\newtheorem{proposition}[theorem]{Proposition}
\newtheorem*{proposition*}{Proposition}
\newtheorem{corollary}[theorem]{Corollary}
\newtheorem*{corollary*}{Corollary}
\newtheorem{lemma}[theorem]{Lemma}
\newtheorem*{lemma*}{Lemma}

\theoremstyle{definition}
\newtheorem{remark}[theorem]{Remark}

\newtheorem*{remark*}{Remark}
\newtheorem{fact}[theorem]{Fact }
\newtheorem*{questions*}{Questions}
\newtheorem*{question*}{Question}

\theoremstyle{definition}
\newtheorem{definition}[theorem]{Definition}
\newtheorem*{definition*}{Definition}
\newcommand{\nc}{\newcommand}

\newcommand{\C}{{\mathbb C}}
\newcommand{\N}{{\mathbb N}}
\newcommand{\Z}{{\mathbb Z}}
\newcommand{\DD}{{\mathbb D}}
\newcommand{\R}{{\mathbb R}}

\newcommand{\T}{{\mathbb{T}}}

\newcommand{\Int}{\operatorname{Int}}

\newcommand{\Hol}{\operatorname{Hol}}

\newcommand{\mult}{\operatorname{mult}}
\newcommand{\dist}{\operatorname{dist}}
\newcommand{\sinc}{\operatorname{sinc}}
\nc{\Intl}{\Int_{l^{\infty}}}

\newcommand{\lra}{\longrightarrow}
\newcommand{\lmto}{\longmapsto}
\newcommand{\eps}{\varepsilon}
\newcommand{\vp}{\varphi}
\nc{\bea}{\begin{eqnarray}}
\nc{\eea}{\end{eqnarray}}
\nc{\beqa}{\begin{eqnarray*}}
\nc{\eeqa}{\end{eqnarray*}}
\nc{\Hi}{H^{\infty}}
\nc{\loi}{\ell^{\infty}}
\nc{\NL}{N^+\vert \Lambda}
\nc{\liL}{\lambda\in\Lambda}
\nc{\nn}{\nonumber}
\nc{\hf}{{\mathcal H}_{\Phi}}
\nc{\hF}{{\mathcal H}_{\Phi}}

\newenvironment{proof*}{\vskip 2mm\noindent {}}{$\blacksquare$ \vskip 2mm}
\numberwithin{equation}{section}

\renewcommand{\Im}{\mbox{Im}}

\nc{\card}{\operatorname{card}}
\nc{\sign}{\operatorname{sign}}
\nc{\spn}{\operatorname{span}}
\nc{\Lin}{\operatorname{Lin}}
\nc{\tsn}{\tilde{\sigma}_n}
\nc{\tsk}{\tilde{\sigma}_k}
\nc{\tskp}{\tilde{\sigma}_k^+}
\nc{\tskm}{\tilde{\sigma}_k^-}
\nc{\D}{\displaystyle}
\nc{\vt}{\vartheta}

\title[Uniform minimality and unconditionality]{Uniform minimality,
  unconditionality and interpolation in backward shift invariant spaces}

\author{Eric Amar \& Andreas Hartmann}

\address{Equipe d'Analyse \& G\'eom\'etrie,
Institut de Math\'ematiques de Bordeaux,
Universit\'e Bordeaux I, 351 cours de la Lib\'eration,
33405 Talence, France}
\email{Eric.Amar@math.u-bordeaux1.fr}
\email{Andreas.Hartmann@math.u-bordeaux1.fr}

\date{\today}



\begin{document}

\begin{abstract}
We discuss relations between uniform minimality, unconditionality
and interpolation
for families of reproducing kernels in backward shift invariant
subspaces. This class of spaces contains as prominent examples
the Paley-Wiener spaces for which it is known that uniform minimality
does in general neither imply interpolation nor unconditionality. 
Hence, contrarily to
the situation of standard Hardy spaces (and other scales of spaces),
changing the size of the space seems in this context necessary to deduce  
unconditionality or interpolation from uniform minimality. 
Such a change can
take two directions: lowering the power of integration, or ``increasing''
the defining inner function (e.g.\ increasing the type in the
case of Paley-Wiener space). 
\end{abstract}

\maketitle

\section{Introduction}

A famous result by Carleson states that a sequence of
points $S=\{a_k\}$ in the unit disk $\DD=\{z\in\C:|z|<1\}$
is an interpolating
sequence for the space $\Hi$ of bounded analytic functions on $\DD$,
meaning that every bounded sequence on $S$ can be interpolated 
by a function $f$ 
in $\Hi$ on $S$, 
i.e.\ $\Hi|S\supset l^{\infty}$,
if and only if the sequence $S$ satisfies the 
Carleson condition:
\bea\label{eq1}
 \inf_{a\in S} |B_{a}(a)|=\delta>0,
\eea
where $B_{a}=\prod_{u\neq a}b_{u}$ is the
Blaschke product vanishing exactly on $S\setminus\{a\}$,
and $b_{a}(z)=\frac{|a|}{a}\frac{a-z}
{1-\overline{a}z}$ (see \cite{carl}). We will write
$S \in (C)$ for short when $S$ satisfies \eqref{eq1}.
Obviously in this situation we also have the embedding
$\Hi|\Lambda\subset l^{\infty}$, so that $S\in (C)$ is
equivalent to $\Hi|\Lambda = l^{\infty}$.
Subsequently it was shown by Shapiro and Shields \cite{SS} 
that for $p\in (1,\infty)$ a similar
result holds:
\beqa
 H^p|S\supset l^p(1-|a|^2)=\{(v_a)_{a\in S}:\sum_{a\in S} (1-|a|^2)
 |v_a|^p<\infty\}
\eeqa
if and only if $S\in (C)$. 
Again, it turns out that we also have $H^p|S\subset l^p(1-|a|^2)$
(the measure $\sum_{a\in S}(1-|a|^2)\delta_a$ is a so-called Carleson
measure), so that $S\in (C)$ is equivalent to
$H^p|S = l^p(1-|a|^2)$.
Considering reproducing kernels
$k_{a}(z)=(1-\overline{a}z)^{-1}$
the interpolation condition and the Carleson condition
can be restated in terms of geometric properties of the
sequence $(k_{a})_{a\in S}$. More precisely the
Carleson condition is equivalent to
$(k_{a}/\|k_{a}\|_{p'})_a$
being uniformly minimal in $H^{p'}$, and the interpolating condition
$H^p|S=l^p(1-|a|)$ to $(k_{a}/\|k_{a}\|_{p'})_{a\in S}$
being an unconditional sequence in $H^{p'}$
(precise definitions will be given below).
Hence, another way of stating the interpolation result in
Hardy spaces is to say that a sequence of normalized reproducing
kernels in $H^{p'}$ is uniformly minimal if and only if it is
an unconditional basis in its span (since interpolation in the 
scale of Hardy spaces does not depend on $p$, the distinction
between $p$ and $p'$ may appear artificial here). This special situation
is not isolated. It turns out to be true in the Bergman
space (see \cite{SchS1}), and in Fock spaces and Paley-Wiener
spaces for certain indices of $p$ (see \cite{SchS2}).

More recently, in \cite{Amar} the first named author has given a method
allowing to deduce interpolation from
uniform minimality when the size of the space is increased
by lowering the power of integration.
This result requires that the underlying space is the closure
of a uniform algebra, and applies in particular to Hardy spaces
on the ball. 

We would like to use some of the methods discussed in \cite{Amar}
to show that uniform minimality implies unconditionality in 
a bigger space for certain backward shift invariant subspaces $K^p_I$ 
for which the Paley-Wiener spaces are a particular instance. 
Recall that for an inner function $I$, $K^p_I=H^p\cap
I\overline{H^p_0}$ (when considered as a space of functions on $\T$), which
is equal to the orthogonal complement
of $IH^2$ when $p=2$. Note also that these spaces are projected
subspaces of $H^p$ ($1<p<\infty$), 
and the projection --- orthogonal when $p=2$ ---
is given by $P_I=IP_-\overline{I}$, where $P_-=Id-P_+$ and
$P_+$ is the Riesz projections of $f(e^{it})=\sum_{n\in\Z}a_ne^{int}\in
L^p(\T)$ onto the analytic part $\sum_{n\ge 0}a_ne^{int}$.
We would like to draw the attention of the reader to
the special situation when $I(z)=I_{\tau}(z):=\exp(2\tau
(z+1)/(z-1))$. Then, the space $K^p_I$ is isomorphic to the Paley-Wiener
space $PW^p_{\tau}$ of entire functions of exponential type $\tau$
and $p$-th power integrable on the real line (see
Section \ref{ss2.3}). By the Paley-Wiener
theorem, $PW^2_{\tau}$ is isometrically isomorphic to $L^2(-\tau,\tau)$.
Already in this ``simple'' case the description of interpolating sequences
is not known (see more comments below). 
There exist sufficient density conditions for interpolation
(or unconditionality) when $p=2$. 
They allow to check that a certain uniform minimal sequence,
which is not unconditional, becomes unconditional
when we ``increase'' the inner function meaning
that we replace $I$ by $I^{1+\eps}$, $\eps>0$. (It is well known
that $K^2_I\subset K^2_{I^{1+\eps}}$ and even $K^2_{I^{1+\eps}}=
\overline{K^2_I+IK^2_{I^{\eps}}}$.)
The density conditions for $p=2$ do not seem to
generalize to $p\neq 2$ (see Proposition
\ref{prop2.4} and comments at the end of Section \ref{ss2.3}), 
so that there is no easy argument that could show that
lowering the integration power without changing $I$ is
sufficient to deduce unconditionality from uniform minimality.
This makes the problem extremely delicate.
So, in the general situation that we consider and where density
or other usable conditions are not known, it seems extremly 
difficult to deduce interpolation from uniform minimality only
be increasing the space in one direction (either adding factors
to $I$ or lowering the integration power $p$). Let us mention
however that under the assumption $I(\lambda_n)\to 0$ the equivalence
between uniform minimality and unconditionality in $K^2_I$ has
been established in \cite{HNP} (see also \cite{Fr} for
a vector valued version of this result).

Our results will 
require some conditions on the inner function such as
being one-component. This means that the level set $L(I,\eps)
=\{z\in \DD:|I(z)|<\eps\}$ of the inner function $I$ is
connected for some $\eps\in (0,1)$ (which is for instance the case
for $I_{\tau}$). One-component inner functions appear in
work by Aleksandrov, Treil-Volberg etc.\ in the connection with
embedding theorems and Carleson measures.

As a consequence of our discussions we state here a sample
result:

\begin{theorem}\label{thm1}
Let $I$ be a one-component singular inner function,
$S\subset \DD$, $1<p\le 2$. Suppose
that $\sup_{a\in S}|I(a)|<1$. If
$(k_{a}^I/\|k_{a}^I\|_{p'})_{a\in S}$
is uniformly minimal in $K^{p'}_I$, where $1/p+1/p'=1$, 
then for every $\eps>0$ and
for every $s<p$, 
$S$ is an interpolating sequence for $K^s_{I^{1+\eps}}$ and
$(k_{a}^I/\|k_{a}^I\|_{s'})_{a\in S}$
is an unconditional sequence in $K^{s'}_{I^{1+\eps}}$,
$1/s+1/s'=1$.
\end{theorem}

As already pointed out, a characterization of
interpolating sequences already for Paley-Wiener spaces is
unknown for general $p$
(when $p=\infty$ Beurling gives a characterization, and
for $0<p\le 1$, see \cite{Fl}; a crucial difference between these
cases and $1<p<\infty$ is the boundedness of the Hilbert transform
on $L^p$). For the case of complete interpolating
sequences in $PW^p_{\tau}$, i.e.\ interpolating sequences for which the
interpolating functions are unique, these are characterized
in \cite{LS} appealing to the Carleson condition and the
Muckenhoupt $A^p$-condition for some function associated with
the generating function of $S$. 
Sufficient conditions are pointed
out in \cite{SchS2} using a kind of uniform zero-set condition in the
spirit of Beurling. Such a condition cannot be necessary since there are
complete interpolating sequences in the Paley-Wiener spaces. 
Another approach is based on invertibility
properties of $P_I|K^p_B$, where $B=\prod_{a\in S}b_{a}$
and discussed in the seminal paper \cite{HNP} (see 
also \cite{Nik}). 
Once having
observed that the Carleson condition for $S$ is necessary
(under the condition $\sup_{a\in S}|I(a)|<1$),
and so $(k_{a}/\|k_{a}\|_p)_{a\in S}$ is an
unconditional basis for $K^p_B$,
the left invertibility of $P_I|K_B^p$ guarantees that
$(k_{a}^I/\|k_{a}^I\|_p)_{a\in S}$ is still an
unconditional sequence. 
The invertibility properties of $P_I|K_B^p$ can be reduced
to the invertibility properties of a certain Toeplitz
operator ($T_{I\overline{B}}$). Again, and also in this approach,
one can feel an essential difference between complete
interpolating sequences and not necessary complete interpolating sequences.
The case of complete interpolating sequences corresponds
to invertibility of $T_{I\overline{B}}$, and a criterion of
invertibility of Toeplitz operators is known. This is the
theorem of Devinatz and Widom (see e.g.\ \cite[Theorem B4.3.1]{Nik}) for
$p=2$ and Rochberg (see \cite{Ro}) for $1<p<\infty$,
and again it is based on the Muckenhoupt
$(A_p)$ condition (or the Helson-Szeg\H{o} condition in
case $p=2$), this time for some function $h\in H^p$
such that $I\overline{B}=\overline{h}/h$.
A useful description of left-invertibility of Toeplitz operators,
the situation corresponding to general not necessarily
complete interpolating sequences, is not available. For the case
$p=2$ an implicit condition is given in \cite{HNP}, and
a condition based on extremal functions in the kernel of the
adjoint $T_{\overline{I}B}$ can be found in \cite{HSS}.


The paper is organized as follows. In the next section we introduce
the necessary material on uniform minimality, dual boundedness 
and unconditionality.
A characterization of unconditional bases of point evaluations
(or reproducing kernels)
will be given in terms of interpolation and embedding. We will
also discuss some Carleson-type conditions which are naturally
connected with embedding problems. Section \ref{ss2.3} is
devoted to a longer discussion of the situation in the Paley-Wiener
spaces. We essentially put the known material in the perspective
of our work. This should convince the reader that it is difficult 
to get better result. In the last section we give our main result
Theorem \ref{thm3.2} which as a special case contains
Theorem \ref{thm1}.




\section{Preliminaries}

\subsection{Geometric properties of families of vectors of
Banach spaces} 

We begin with some observations in the classical $H^p$
concerning the relation between uniform minimality and
unconditionality.
Recall that the reproducing kernel of $H^p$ in $a\in \DD$
is given by $k_{a}(z)=(1-\overline{a}z)^{-1}$.
The Carleson condition $\inf_{a\in S}|B_a (a)|\ge \delta>0$ 
can then be restated as 
$(k_{a}/\|k_{a}\|_{p'})_{a\in S}$ being a uniformly minimal 
sequence in $H^{p'}$ (which is equivalent here to
$(k_{a}/\|k_{a}\|_{p})_{a\in S}$ being uniformly minimal
in $H^p$). Let us explain this a little bit more.
By definition a sequence of normalized vectors $(x_n)_n$
in a Banach space $X$ is uniformly minimal if
\bea\label{unifmin}
 \inf_n \dist (x_n,\bigvee_{k\neq n}x_k)=\delta>0.
\eea
(Here $\bigvee_i x_i$ denotes the closed linear span of the
vectors $x_i$.) By the Hahn-Banach theorem this is equivalent
to the existence of a sequence of functionals $(\vp_n)_n$ in
$X^*$ such that $\vp_n(x_k)=\delta_{nk}$, where $\delta_{nk}$ is
the usual Kronecker symbol, and $\sup_n\|\vp\|_{X^*}<\infty$. 
In our situation, setting 
\beqa
 {\vp}_a= \frac{B_{a}}{B_{a}(a)}
 \frac{k_{a}}{k_{a}(a)}\|k_a\|_{p},
\eeqa
we get
\beqa
 \langle {\vp_a},\frac{k_{b}}{\|k_{b}\|_p}\rangle
 =\delta_{ab}.
\eeqa
Since $\|k_{a}\|_s\simeq (1-|a|^2)^{1-1/s}$ we moreover have
$\sup_{a\in S}\|\vp_a\|_q< \infty$.
Another way of viewing the uniform minimality condition when $p=2$
is given in terms of angles: a sequence $(x_n)_n$ of vectors
in a Hilbert space 
is uniformly minimal if the angles between $x_n$ and $\bigvee_{k\neq n}
x_k$ are uniformly bounded away from zero.

A notion closely related with uniform minimality is that
of dual boundedness (see \cite{Amar}). Let us give a general
definition
\begin{definition}\label{defdb}
Let $X\subset \Hol(\Omega)$ be a reflexive Banach space of 
holomorphic functions
on a domain $\Omega$. Suppose that the point evaluations $E_z$ are
continuous for every $z\in \Omega$. 
A sequence $S\subset \Omega$ is called dual-bounded if
the sequence $(E_a/\|E_a\|_{X^*})_{a\in S}$ of reproducing kernels
is uniformly minimal.
\end{definition}

Again, by the Hahn-Banach theorem this means that there exists a
sequence $(\rho_a)_{a\in S}$ of elements in $X$ ($=X^{**}$)
with uniformly bounded norm $\sup_{a\in S}\|\rho_a\|_X<\infty$
and $\langle \rho_a,E_b/\|E_b\|_{X^*}\rangle=\delta_{ab}$, i.e.\
$\rho_a(b)=\delta_{ab}\|E_b\|_{X^*}$.

This condition is termed weak interpolation in \cite{SchS2}.

Let us discuss the unconditionality. Recall that a
basis $(x_n)_n$ of vectors in a Banach space $X$ is an
unconditional basis if for every $x\in X$, there exists a
numerical sequence $(\alpha_n)$ such that the sum
$\sum_n \alpha_n x_n$ converges to $x$, and for every
sequence of signs $\eps=(\eps_n)$, the sum $\sum_n\eps_n\alpha_n x_n$
converges in $X$ to a vector $x_{\eps}$ with norm comparable to $\|x\|$.
We will discuss the interpolation condition $H^p|\Lambda\supset
l^p(1-|a|^2)$ in the light of this definition using reproducing kernels.
First recall from \cite{SS} that we have
$H^p|\Lambda=l^p(1-|a|^2)$.
Let $B=B_{S}$ be the
Blaschke product vanishing on $S$. Set $K^p_B=H^p
\cap B\overline{H^p_0}$, where $H^p_0=zH^p$. The space $K^p_B$
is a backward shift invariant subspace. Also, $K^p_B
=\bigvee_{a\in S} k_{a}$, and $H^p=K^p_B+BH^p$ ($K^p_B=P_BH^p$ is
a projected space).
So the interpolation condition is equivalent
to $K^p_B|\Lambda=l^p(1-|a|^2)$, and since the interpolation
problem has unique solution in $K^p_B$, we have 
for every $f\in K^p_B$, $\|f\|_p^p\simeq \sum_n (1-|a|^2)
|f(a)|^p$. Clearly under this
condition the functions $\vp_a$ introduced above exist and 
are in $K^p_B$. Then for every finite sequence $(v_a)$ 
and every  sequence of signs $(\eps_n)$ we
have
\beqa
 \|\sum_{a\in S} \eps_a v_a\vp_a\|_p^p
 \simeq \sum_{a\in S} (1-|a|^2)|\eps_a|^p|v_n|^p
 \simeq \sum_{a\in S} (1-|a|^2)|v_n|^p 
\eeqa
which shows that $(\vp_a)_a$ is an unconditional basis in $K^p_B$.
Then $(k_{a}/\|k_{a}\|_p)$ is also an unconditional
basis in $K^p_B$.

Again, the unconditionality can be expressed in terms of angles
when $p=2$: a sequence $(x_n)_n$ of vectors in a Hilbert space
is unconditional if the angles between $\bigvee_{k\in \sigma}x_k$
and $\bigvee_{k\in \N\setminus\sigma} x_k$ is uniformly bounded
away from zero for every $\sigma\in \N$.

So the interpolation results tell us that in $H^p$ a sequence
of reproducing kernels is uniformly minimal if and only if
it is an unconditional sequence. Such results also hold in
other spaces like e.g.\ Bergman spaces (see \cite{SchS1})
and in Fock and Paley-Wiener spaces for certain values of $p$
(see \cite{SchS2}).

We will be interested in the situation in backward shift invariant
subspaces $K^p_I$.

\subsection{Unconditional bases and interpolation}

In this section we will establish a general link between
unconditional basis on the one hand and interpolation with
an additional embedding property on the other hand. It turns
out that this link can be reformulated, in the spirit of 
\cite[Theorem1.2]{ni78}, in abstract
terms without appealing to the notion of interpolation.
We will start with this general result before coming back
to the special context of interpolation.

Suppose that $X$ is a reflexive Banach space, and 
let $(y_n)_n$ be a sequence of normalized elements in
$X^*$ that we suppose at least minimal:
$\dist (y_n,\bigvee_{k\neq n}y_k)>0$ for every $n\in \N$. 
We set $Y=\bigvee y_n$ and $N:=Y^{\perp}\subset (X^*)^*=X$. 
By the minimality condition
there exists a sequence $(x_n)_n\in X^{**}=X$ such that 
$\langle x_n,y_k\rangle_{X=X^{**},X^*}=\delta_{n,k}$, $n,k\in \N$.

For a sequence space $l$, we consider the canonical system
$\{e_n\}_n$ where $e_n=(\delta_{nk})_k$. The space $l$ will be
called ideal if whenever $(a_n)_n\in l$ and $|b_n|\le |a_n|$,
$n\in \N$, then $(b_n)\in l$.
Recall also that a family of vectors in a Banach space is called
fundamental if it generates a dense set in the Banach space.
Observe that the canonical system is an unconditional basis in $l$
if and only if $l$ is ideal and the canonical system is fundamental
in $l$.

We obtain the following result.

\begin{proposition}
Let $X$ be a reflexive Banach space.
With the above notation, the following assertions are
equivalent.
\begin{enumerate}
\item The sequence $(y_n)_n$ is an unconditional basis in 
$Y=\bigvee_n y_n$.
\item The sequence 
$(x_n+N)_n$ 
is an unconditional basis in $X/N$ (in general not normalized).
\item There exists two reflexive 
Banach sequence spaces $l_1$, $l_2$, in which
the respective canonical systems are unconditional bases
and such that
\begin{itemize}
\item[(i)] The set of generalized Fourier coefficients of $X$ contains
$l_1$:
\beqa
 \{\langle x,y_n \rangle_{X,X^*})_n:
  x\in X\}\supset l_1,
\eeqa
\item[(ii)] for every $\mu=(\mu_n)_n\in l_2$,
\beqa
 \|\sum_n \mu_n y_n\|_{X^*}\lesssim \|\mu\|_{l_2};
\eeqa
\end{itemize}
moreover $l_2\simeq l_1^*$ and the duality of $l_1$ and $l_1^*\simeq
l^2$ is given by $\langle (\alpha_n)_n,(\mu_n)_n\rangle_{l_1,l_2}
=\sum_n \alpha_n\mu_n$.
\end{enumerate}
\end{proposition}

This theorem is in the spirit of \cite[Theorem 1.2]{ni78}. However,
in Nikolski's theorem there does not really appear
the condition (i) together with an embedding of type (ii). 
The condition (i) will later on play the r\^ole of the interpolation
part.

\begin{proof}. Observe first that $Y^*=(X^*)^*/Y^{\perp}=X/N$.
Moreover, for every $u\in N=Y^{\perp}$, $\langle x_n+u,y_k\rangle
=\langle x_n,y_n\rangle=\delta_{nk}$, and hence 
$((y_n)_n,(x_n+N)_n)$ is a biorthogonal system in $(Y,Y^*)$.
By the general theory (see for instance \cite[Corollary I.12.2,
Theorem II.17.7]{sing}) we obtain the equivalence of (1) and (2).

Let us now prove that (1) and (2) imply (3).
By \cite[Theorem 1.1]{ni78} the sequence $(y_n)_{n}$ 
is an unconditional sequence in $Y$ 
if and only if the multiplier space $\mult(y_n):=\{\mu=(\mu_n)_{n}:
T_{\mu}:\Lin (y_n)\lra \Lin(y_n)$, $\sum_{finite}\alpha_ny_n
\lmto \sum_{finite}\mu_n\alpha_n y_n$ extends to a bounded
operator on $Y\}$ is equal to $l^{\infty}$.
And this, by \cite[Lemma 1.2]{ni78} is equivalent to the
existence of a sequence space $l_2$ in which the canonical system
is an unconditional basis 
such that $(y_n)_{n}$ is a $l_2$-basis, which means that 
\beqa
 T:l_2&\lra&Y\\
 (\mu_n)_{n}&\lmto& \sum_{n}\mu_n y_n
\eeqa
is an isomorphism. Note that $Y$
is reflexive as a closed subspace of the reflexive Banach space
$X^*$, and so is $l_2$.

For exactly the same reason, by (2) there exists a sequence space $l_1$ 
with the required properties such that
\beqa
 S:l_1&\lra&X/N\\
 (\alpha_n)_{n}&\lmto& \sum_{n}\alpha_n x_n=:x_{\alpha}+N
\eeqa
is an isomorphism. Note that $X/N$
is reflexive as a quotient space of the reflexive Banach space
$X$, and so is $l_1$. Take $(\alpha_n)_n\in l_1$, then
$S((\alpha_n)_n)=x_{\alpha}+N\in X/N$ for a suitable $x_{\alpha}\in X$.
Now $(\langle x_{\alpha},y_n\rangle)_n
=(\langle \sum_k\alpha_kx_k,y_n\rangle)_n=
(\alpha_n)_n$ (note that
$\sum_k\alpha_kx_k+N$ 
converges in $X/N$).
So $(\alpha_n)_n\in \{(\langle x,y_n\rangle):x\in X\}$.

Finally, since $l_1\simeq X/N$, $l_2\simeq Y$ and $Y^*=X/N$ we have
$l_2^*\simeq l_1$ and by reflexivity $l_2\simeq l_1^*$.
Moreover,
by the idenfication maps we can write for $(\alpha_n)_n\in l_1$
and $(\mu_n)_n\in l_2\simeq l_1^*$:
\beqa
 \langle (\alpha_n)_n,(\mu_n)_n\rangle_{l_1,l_2}
 =\langle \sum_n \alpha_n x_n+N,\mu_k y_n\rangle_{X/N,Y}
 =\sum_{n,k}\alpha_n\mu_k \langle x_n,y_n\rangle_{X,Y}
 =\sum_n\alpha_n\mu_n.
\eeqa

We finish by showing that (3) implies (1).
By (ii), the operator $T$ is bounded and by construction
onto, so that we are done if we can
show that $T$ is left invertible: $\|\mu\|_{l_2}\lesssim \|T\mu\|_Y$.
Now by (i) for $(\alpha_n)_n\in l_1$, there exists $x_{\alpha}\in X$
such that $\alpha_n=\langle x_{\alpha},y_n\rangle$.
Let us introduce the operator
\beqa
 A:l_1&\lra& X/N\\
 (\alpha_n)_n&\lmto&x_{\alpha}+N.
\eeqa
This operator is well defined (if we choose $x_{\alpha}'$ with
$\langle x_{\alpha}',y_n\rangle=\alpha_n$, then
$\langle x_{\alpha}'-x_{\alpha},y_n\rangle=0$ for every $n$ and
$x_{\alpha}'-x_{\alpha}\in N$). It is also linear. Let us check that
its graph is closed. For this consider a sequence $(\alpha^N_n)_n$
converging to $(\alpha_n)_n$ in $l_1$. Since the canonical basis is
an unconditional basis in $l_1$, we obtain coordinate-wise convergence:
$\alpha_n^N\to\alpha_n$ when $N\to\infty$. We assume that 
$A((\alpha_n^N)_n)=x_{\alpha^N}+N\to x+N$. Note that $A((\alpha_n)_n)=
x_{\alpha}+N$. Then for every $n$ we have $\langle x,y_n\rangle
=\lim_{N\to\infty}\langle x_{\alpha^N},y_n\rangle
=\lim_{N\to\infty} \alpha_n^N=\alpha_n=\langle x_{\alpha},y_n\rangle$. So 
$x-x_{\alpha}\in N$ and $x+N=A((\alpha_n)_n)$. By the closed graph
theorem $A$ is bounded.

Let us show that $A^*:(X/N)^*=Y\to l_1^*$ is the left inverse
to $T$ (modulo the isomorphism from $l_1^*$ to $l_2$).
Equivalently it is sufficient to show that $T^*A:l_1\to l_2^*$
is an isomorphism. Note that for $(\alpha_n)_n\in l_1$
and $(\mu_n)_n\in l_2$, we have
\beqa
 \langle T^*A(\alpha_n)_n,(\mu_n)_n\rangle_{l_2^*,l_2}
 &=&\langle A(\alpha_n)_n,T(\mu_n)_n\rangle_{X/N,Y}
 =\langle x_{\alpha}+N,\sum \mu_n y_n\rangle_{X/N,Y}\\
 &=&\sum_n \mu_n\langle x_{\alpha},y_n\rangle_{X,X^*}\\
 &=&\sum_n\mu_n\alpha_n
\eeqa
By assumption this is equal to  
$\langle (\alpha_n)_n,(\mu_n)_n\rangle_{l_1,l_2}$ so that
for every $(\alpha_n)_n\in l_1$
and $(\mu_n)_n\in l_2$, we have
\beqa
 \langle T^*A(\alpha_n)_n,(\mu_n)_n\rangle_{l_2^*,l_2}
 =\langle (\alpha_n)_n,(\mu_n)_n\rangle_{l_1,l_2}.
\eeqa
Hence $T^*A$ is the identity (modulo the identification between
$l_1$ and $l_2^*$).
\end{proof}

It is interesting to note that when $X$ is a Hilbert space more can be said
about the structure of $l$:
it is clear that then $l=l^2$. 
However, by a result of
Lindenstrauss and Zippin (see \cite{LZ}), if in a
Banach space $X$ every two normalized unconditional bases
are isomorphic to each other, then $X$ is isomorphic to
one of the following spaces $c_0$, $l^1$ or $l^2$. In
other words the general theory does not yield $l=l^p$
when $(x_n)_n$ is an unconditional basis in
(a subspace of) $X=L^p$ (Pelczynski constructed actually unconditional
bases in $l^p$ which are not equivalent to the canonical
basis, \cite{pel}).

Let $X$ be Banach space of holomorphic functions on $\DD$,
such that 
the point evaluations $E_a$ in $a\in \DD$ are continuous
in $X$. 
A sequence $S\subset \DD$
is called $l$-interpolating for a sequence space $l$ (defined
on $S$) if for every sequence $v=(v_a)_{a\in S}$ with
$(v_a/\|E_a\|_{X^*})_{a\in S}\in l$ there is a function $f\in X$
with $f(a)=v_a$, i.e.\ 
\beqa
 X|S\supset l(1/\|E_a\|_{X^*}):=\{v=(v_a)_{a\in S}:
 (v_a/\|E_a\|_{X^*})_{a\in S}\in l\}.
\eeqa
Since, $\|E_a\|_{(H^p)^*}\simeq \|k_a\|_{p'}\simeq
(1-|a|^2)^{-1/p}$ ($1<p<\infty$), this definition is consistent with
the definitions we gave before for $H^p$, in which case we had
chosen $l=l^p$.

The reader should also note that in the previous subsection we 
have repeatedly used the fact that interpolation in $H^p$, i.e.\ 
$H^p|S\supset l^p(1-|a|^2)$ 
(we will not consider the case $p=\infty$ here)
implies in fact the equality 
$H^p|S= l^p(1-|a|^2)$ (this is Shapiro and Shields' result, \cite{SS}).

In the general case, without any further information, we have to 
impose an additional embedding.
For the convenience of the proof
in the following result we will suppose that $X$ is reflexive
(and so $l$ will be). We will also need the notion of ideal space.
A sequence space $l$ is called ideal if whenever $v=(v_n)_n\in l$
and $w=(w_n)_n$ is any numerical sequence with $|w_n|\le |v_n|$
for every $n$ then also $w\in l$. This notion appears naturally
in the context of free interpolation and unconditional bases.

\begin{proposition}\label{prop2.2}
Suppose $X\subset \Hol(\DD)$ is reflexive and $S$ is a sequence in $\DD$.
The following assertions are equivalent.
\begin{itemize}
\item [(1)] There exists a reflexive and ideal sequence space $l$ 
such that
\begin{itemize}
\item[(i)] $S$ is $l$-interpolating
\item[(ii)] There is a constant $C$ such that for every finitely supported 
sequence $\mu=(\mu)_{a\in S}$, we have
$\|\sum_{a\in S}\mu_a \frac{\D E_a}
{\D \|E_a\|_{X^*}}\|_{X^*}\le C \|\mu\|_{l^*}$,
\end{itemize}
\item [(2)]  $(E_a)_{a\in S}$ is an unconditional sequence in $X^*$.
\end{itemize}
\end{proposition}

A sequence satisfying condition
(ii) will be called $l^*$-Carleson or $q$-Carleson when
$l^*=l^q$ (a Carleson embedding for $X^*$ with respect to the
sequence space $l^*$). See Subsection \ref{ss2.4} for
more on Carleson conditions. 

Note that another way of writing (ii) is
\beqa
 \forall f\in X,\forall \mu\in l^*,\quad
 |\sum_{a\in S}\mu_a \frac{f(a)}{\|E_a\|_{X^*}}|\le  C
 \|f\|_X\|\mu\|_{l^*},
\eeqa
which means that for every $f\in X$, the sequence 
$(f(a)/\|E_a\|_{X^*})_{a\in S}$ is in  $(l^*)^*=l$, and hence
(ii) is equivalent to 
\bea\label{lcarl}
 \|(f(a)/\|E_a\|_{X^*})_{a\in S}\|_l\le C \|f\|_X,
\eea
which means $X|\Lambda\subset l(1/\|E_a\|_{X^*})$ (there will be
more discussions on Carleson measures in Subsection \ref{ss2.4}).
We thus have

\begin{corollary}\label{cor2.3new}
Suppose $X\subset \Hol(\DD)$ is reflexive and $S$ is a sequence in $\DD$.
The following assertions are equivalent.
\begin{itemize}
\item [(1)] There exists a reflexive and ideal sequence space $l$ 
such that $X|\Lambda=l(1/\|E_a\|_{X^*})$
\item [(2)]  $(E_a)_{a\in S}$ is an unconditional sequence in $X^*$
(an $l^*$-basis in its span).
\end{itemize}
\end{corollary}

\begin{proof}[Proof of Proposition \ref{prop2.2}]
By \cite[Theorem 1.1]{ni78} the sequence $(E_a)_{a\in S}$ 
is an unconditional sequence in its span if and only if
the multiplier space $\mult(E_a):=\{\mu=(\mu_a)_{a\in S}:
T_{\mu}:\Lin (E_a)\lra \Lin(E_a)$, $\sum_{finite}\alpha_aE_a
\lmto \sum_{finite}\mu_a\alpha_a E_a$ extends to a bounded
operator on $X_0^*:=\bigvee_{a\in S} E_a\}$ is equal to $l^{\infty}$.
And this, by \cite[Lemma 1.2]{ni78} is equivalent to the
existence of an ideal space $l_0$ such that $(E_a)_{a\in S}$
is a $l_0$-basis, which means that 
$X_0\simeq l_0(E_a):=\{(\alpha_a)_{a\in S}:
(\alpha_a \|E_a\|_{X^*})\in l_0\}$, in other words
the mapping $(\alpha_a)_{a\in S}\lmto \sum_{a\in S}\alpha_a E_a$
is an isomorphism from $l_0(E_a)$ onto $X_0$, or
equivalently 
\beqa
 T:l_0&\lra&X^*\\
 (\beta_a)_{a\in S}&\lmto& \sum_{a\in S}\beta_a 
 \frac{E_a}{\|E_a\|_{X^*}}
\eeqa
is an isomorphism. Note that $X_0$
is reflexive as a closed subspace of a reflexive Banach space,
and so is $l_0$.
Set $l:=l_0^*$ (so that $l^*=l_0$). 
By the preceding argument, $(E_a)_{a\in S}$ is
an unconditional sequence in its span if and only if
\bea\label{equinorm}
 c \|\mu\|_{l^*}\le \|\sum_{a\in S}\mu_a \frac{E_a}{\|E_a\|_{X^*}}
  \|_{X^*}\le C \|\mu\|_{l^*},
\eea
for some fixed constants $c$, $C$.
This yields
in particular (ii). 

We will compute the adjoint operator
$T^*:X\lra l$. 
Let $\mu\in l^*$,
\beqa
 \langle T^*f,\mu \rangle=
 \langle f,T\mu\rangle=\langle f,\sum_{a\in S}\mu_a
 \frac{E_a}{\|E_a\|_{X^*}}\rangle
 =\sum_{a\in S}\mu_a \frac{f(a)}{\|E_a\|_{X^*}}.
\eeqa
Hence,
the functional $T^*f$ on $l^*$
is represented by a
sequence the entries of which are given by $f(a)/\|E_a\|_{X^*}$, $a\in S$.
In other words
$T^*f=(f(a)/\|E_a\|_{X^*})_{a\in S}\in (l^*)^*=l$.

Now the left hand inequality in \eqref{equinorm} is equivalent 
to the left invertibility of $T$ which is equivalent to the
surjectivity of $T^*$ i.e.\ to the fact  $S$ is
$l$-interpolating.
This show that (2) implies (1).

For the converse implication, note that (ii) implies
the right inequality in \eqref{equinorm}. 
Moreover this inequality shows also that
$T$ is well defined and bounded. 
By the above arguments the surjectivity of $T^*$ is
equivalent to the fact that $S$ is interpolating. On the other hand the
surjectivity of $T^*$ is equivalent to the left invertibility of $T$
and so to the left inequality in \eqref{equinorm}.
\end{proof}


Still the following is true

\begin{corollary}\label{cor2.3}
If $S$ is interpolating for $K^p_I$ and if
there is a constant $C$ such that for every finitely supported 
sequence $\mu=(\mu)_{a\in S}$, we have
$\|\sum_{a\in S}\mu_a k^I_a/\|k_a^I\|_{p'}\|_{p'}\le C \|\mu\|_{l^{p'}}$,
then $(k^I_a/\|k^I_a\|_{p'})_{a\in S}$ 
is an unconditional sequence in $K^{p'}_I$.
\end{corollary}

More precisely the conclusion would be that $(k^I_a)_{a\in S}$
is an $l^{p'}$-basis in its span.
This conclusion can in general not be deduced only from the
condition of unconditionality as explained above. However,
in the special situation $\sup_{a\in S}|I(a)|<1$, 
\cite[Theorem 6.3, Partie II]{HNP} shows that if the reproducing
kernels form an unconditional sequence in $K^{p'}_I$
then automatically they form an $l^{p'}$-basis in their span.

\subsection{Carleson measures}\label{ss2.4}


Let us fix the framework of this subsection.
$S$ is a sequence in $\DD$, $I$ an inner function and
$1\le q <\infty$. For $a\in S$ we denote by
$k_{q,a}^I=k_a^I/\|k_a^I\|_q$
the normalized reproducing kernel.

Let $1\le q<\infty$. Recall that a sequence $S$ is called $q$-Carleson
if 
\beqa
 \exists D_q>0, \forall \mu\in l^q,\ \left\|\sum_{a\in S}
 \mu_a k^I_{q,a}\right\|_q\le D_q \|\mu\|_q.
\eeqa

We will also use the notion of weak $q$-Carleson sequences:

\begin{definition}
Let $2\le q<\infty$. The sequence $S$ is called
weakly $q$-Carleson if
\beqa
 \exists D_q>0, \forall \mu\in l^q,\ \left\|\sum_{a\in S}
 |\mu_a|^2 |k^I_{q,a}|^2 \right\|_{q/2}\le D_q \|\mu\|_q^2.
\eeqa
\end{definition}

Note that by \cite[Lemma 3.2]{Amar}, the $q$-Carleson property
implies the weak $q$-Carleson property.

Observe also that $(l^q)^*=l^p$, that the dual of $K^q_I$ can 
be identified with $K^p_I$, and that the functional of point
evaluation $E_a$ can then be identified with $k^I_a$.
Now, using the notation from the preceding subsection,
by \eqref{lcarl}, $S$ is $q$-Carleson if and only if
for every $f\in K^p_I$, 
\beqa
 \sum_{a\in S}\frac{|f(a)|^p}{\|k_a^I\|_p^p}
 \le c \|f\|_p,
\eeqa
which means that $\nu:=\sum_{a\in S}\delta_a/\|k_a^I\|_p$ 
is a $K^p_I$-Carleson measure:
$K^p_I\subset L^p(\nu)$. 

In the special situation when $I$ is one-component, then
by a result by Aleksandrov (see \eqref{Aleks}), we have
\beqa
 \|k_a^I\|_p\simeq \left(\frac{1-|I(a)|^2}{1-|a|^2}\right)^{1/p'},
\eeqa
$p'$ being the conjugated index to $p$, and so, if $S$ is
$q$-Carleson and $I$ is one-component, then the measure
\beqa
 d\nu =\sum_{a\in S}\frac{1-|a|^2}{1-|I(a)|^2}\delta_a
\eeqa
is $K^p_I$-Carleson.

\noindent {\bf Geometric Carleson conditions}

In \cite{TV}, the following geometric notion of Carleson measure
appears. For an inner function $I$ and an $\eps>0$, 
let $L(I,\eps)=\{z\in \DD:|I(z)|<\eps\}$ be the associated
level set. In the notation of \cite{Al1},
let $\mathcal{C}(I)$ be the set of measures
for which there exists $C>0$ such that
\bea\label{geocarl}
 |\mu|(S(\zeta,r))\le C r
\eea
for every Carleson window $S(\zeta=e^{it},h):=\{z=r e^{i\theta}\in \DD:
1-h<r<1, |t-\theta|<h\}$ meeting $L(I,1/2)$ (this is of course
a weaker notion than the usual one requiring \eqref{geocarl}
on all Carleson windows; the value $\eps=1/2$ is of no
particular relevance). Let also $\mathcal{C}_p(I)$ be the
set of measures for which $K^p_I\subset L^p(\mu)$. Strengthening
the results of \cite{TV}, Aleksandrov proved in \cite[Theorem 1.4]{Al1}
that for one component inner functions $\mathcal{C}(I)=
\mathcal{C}_p(I)$. In other words, the geometric Carleson
condition \eqref{geocarl} on Carleson windows meeting the
level set $L(I,1/2)$ characterizes the $K^p_I$-Carleson
measures for one component inner functions.

Combining these observations, 
we get 
the following characterization. 
\begin{fact}\label{fact1}
Let $I$ be a one-component inner function. Then
the following assertions are equivalent.
\begin{itemize}
\item[(i)] $S$ is $p'$-Carleson
\item[(ii)] $\nu=\sum_{a\in S}\frac{\D 1-|a|^2}{\D 1-|I(a)|^2}\delta_a$
is $K^p_I$-Carleson
\item[(iii)] $\nu$ (as defined in point (ii)) satisfies the geometric
Carleson condition \eqref{geocarl} on Carleson windows meeting
the level set $L(I,1/2)$.
\end{itemize}
\end{fact}


\begin{question*}
Do there exist in backward shift invariant subspaces
interpolating sequences $S$ that are not $p'$-Carleson?
\end{question*}

\section{Paley-Wiener spaces}\label{ss2.3}

We will discuss a special class of backward shift invariant
subspaces. Let $I(z)=e^{i2\pi z}$ be the singular inner function
in the upper half plane with sole singularity at $\infty$
(to fix the ideas, we have chosen the mass of the associated singular
measure to be $2\pi$).
Recall (see \cite[B.1]{Nik}) that the transformation
\beqa
 U_p:H^p(\DD)&\lra&H^p(\C^+)\\
  f&\lmto& \left \{x\to (U_pf)(x)=\left(\frac{1}{\pi(x+i)^2}\right)^{1/p}
 f\left(\frac{x-i}{x+i}\right)\right\}
\eeqa
is an isomorphism of the Hardy space on the disk $H^p(\DD)$ onto
the Hardy space $H^p(\C^+)$ of the upper half plane $\C^+
=\{z\in \C:\Im z>0\}$.
This transformation sends the inner function $I_0(z)
=\exp(2\pi(z+1)/(z-1))$ on $\DD$ to $I$ on $\C^+$.

Let $PW^p_{\pi}$ be the Paley-Wiener space of entire functions
of type $\pi$ which are $p$-th power integrable on the real line.
Pick $f\in PW^p_{\pi}$.
By a  theorem by Plancherel and P\'olya
(see \cite[Lecture 7, Theorem 4]{lev}) we get
\bea\label{PP}
 \int_{\R}|f(x+ia)|^pdx\le e^{p\pi|a|}\|f\|_p^p
\eea
for every $a\in \R$. Setting $F(z)=e^{i\pi z}f(z)$ (which means
that in a sense we compensate the type in the positive imaginary
direction) yields
\beqa
 \int_{\R}|F(z+iy)|^pdx=\int_{\R}|f(x+iy)|^pe^{-p\pi y}dx
 \le \|f\|_p^p
\eeqa
in particular for every $y>0$ which means that $F\in H^p(\C^+)$.
Dividing $F$ by $I$ we obtain an analytic function
in the lower halfplane $\C_-$ and for every $y<0$,
\beqa
 \int_{\R}|F(x+iy)e^{-i2\pi (x+iy)}|^pdx=\int_{\R}|f(x+iy)|^pe^{p\pi y}dx
 \le \|f\|_p^p
\eeqa
so that $F/I$ is in the Hardy space of the lower halfplane 
$H^p(\C_-)$. Hence $F\in H^p(\C^+)\cap \overline{I}H^p_0(\C_-)=:
K^p_{\R,I}$ (now
considered as a space of functions on $\R$, the elements of which can of
course be continued analytically to the whole plane).
It is clear that $K^p_{\R,I}$ can be identified via $U_p$ with
$K^p_I$ on $\DD$ (or $\T$). Hence there is a natural identification
between Paley-Wiener spaces and backward invariant subspaces
(on $\T$ or $\R$): $PW^p_{\pi}=e^{-i\pi z}U_pK^p_I$.

It is well known that in the particular case $p=2$, $PW^p_{\pi}$  
is nothing but $\mathcal{F}L^2(-\pi,\pi)$ (this comes from the Paley-Wiener
theorem).

Let us make another observation concerning 
imaginary translations. For $a\in \R$, let 
\beqa
 \Phi_a:PW^p_{\pi}&\lra& PW^p_{\pi}\\
   f&\lmto& \{\Phi_a f:z\lmto f(z-ia)\}.
\eeqa
Using again the Plancherel-P\'olya theorem (see \eqref{PP}), we see
that $\Phi_a$ is well-defined and bounded (it is clearly linear).
It is also invertible with inverse $\Phi_a^{-1}=\Phi_{-a}$. So
$\Phi_a$ is an isomorphism of $PW^p_{\pi}$ onto itself (the
type that we fixed to $\pi$ here does not really matter).

So the Paley-Wiener spaces are special candidates of our
spaces $K^p_I$, which motivates the following important observations.
In general
it is not true that uniform minimality implies 
interpolation or unconditionality
which we will explain now following
\cite{SchS2}. 

By definition a sequence $\Gamma=\{x_k+iy_k\}_k$ is interpolating
for $PW^p_{\tau}$ if for every numerical sequence $(v_k)_k$ with
\bea\label{eq3.2}
 \sum_k |v_k|^p e^{-p\tau|\eta_k|} (1+|\eta_k|)<\infty
\eea
there exists $f\in PW^p_{\tau}$ with $f(\gamma_k)=a_k$.

\begin{theorem}[Schuster-Seip, 2000]\label{thm3.1}
Let $2\le p <\infty$. 
There exists a dual bounded sequence $\Gamma$ which is not
interpolating in $PW^{p}_{\pi}$. 
\end{theorem}

We would like to recall here the construction of Schuster and
Seip since it will serve later on.

\begin{proof}
Define a squence $\Gamma=\{\gamma_k\}_{k\in\Z}$ by
$\gamma_0=0$ and $\gamma_k(p)=k+\delta_k(p)$, $k\in \Z\setminus\{0\}$,
where $\delta_k(p)=\sign(k)/(2p_0)$ and $p_0=\max(p,p')$,
$1/p+1/p'=1$. Since this sequence is real, the weight appearing
in \eqref{eq3.2} is equal to 1.

Now let $G(z)=z\prod_{k\neq 0}(1-\frac{\D z}{\D \gamma_k})$
which defines
an entire function of exponential type 
$\pi$ with $|G(x)|\simeq d(x,\Gamma) (1+|x|)^{-1/p_0}$.
Note that the $p$-th power integrability of $|G|$ on $\R$ is
determined by $(1+|x|)^{-1/p_0}$, and the latter function is
never $p$-th power integrable on $\R$
(one could distinguish the case $p>2$ and $p<2$).
Hence, $\Gamma$ is a uniqueness set and thus interpolating if
and only if it is completely interpolating. 
  
We will use the same type of computations as in the proof
of \cite[Theorem 2]{LS} to check that $\Gamma$ is not
(completely) interpolating when $p\ge 2$.
According to \cite[Theorem 1]{LS}, it suffices to check that
$F^p$, where
$F(x)=|G(x)/d(x,\Gamma)|\simeq (1+|x|)^{-1/p_0}$, is not
$(A_p)$, i.e.
\beqa
 \frac{1}{|I|}\int_I F^p dt \left(\frac{1}{|I|}\int_I
 F^{-p'}dt \right)^{p-1}
\eeqa
is not uniformly bounded in the intervals $I$.
For $p\ge 2$, we have $p_0=p$ and hence we have to consider
\beqa
  \frac{1}{|I|}\int_I (1+|t|)^{-1} dt \left(\frac{1}{|I|}\int_I
 (1+|t|)^{p'/p}dt \right)^{p-1}.
\eeqa
This expression behaves like $\log (1+|x|)$ when $I=[0,x]$, which
is incompatible with the $(A_p)$-condition. So the sequence $\Gamma$
is not interpolating.

On the other hand, $g_k(z)=G(z)/(z-\gamma_k)$ 
vanishes on $\Gamma\setminus\{\gamma_k\}$ and satisfies
\bea\label{eq2.3.1}
 |g_k(\gamma_k)|\simeq \|g_k\|_{L^p(\R)}. 
\eea
Note that $G\in L^p$
(if and) only if $(1+|x|)^{-1/p'}\in L^p$, i.e. $p/p'=p-1>1$ or
$p>2$. 
This implies that the sequence is dual bounded. In fact, note
that the reproducing kernel of the Paley-Wiener space $PW^p_{\pi}$ in $x\in\R$
is given by $k_x(z)=\sinc (\pi (z-x))=\sin(\pi(z-x))/(\pi(z-x))$,
the norm of which in $L^{p'}(\R)$
can be easily estimated to be comparable to a
constant independantly of $x$.
Hence \eqref{eq2.3.1} implies that $\tilde{g}_k:=g_k/\|g_k\|_p$ is
of uniformly bounded norm and $|\tilde{g}_k(\gamma_k)|\simeq 1
\simeq \|k_{\gamma_k}\|_{L^{p'}(\R)}$. Suitably renormed,
$(\tilde{g}_k)_k$ thus furnishes the family $(\rho_{\gamma_k})_k$
mentioned after Definition \ref{defdb}.
\end{proof}

As a consequence, in $PW^p_{\pi}$ there exists a sequence $\Gamma$ such that
$\{k_{\gamma_l}/\|k_{\gamma_l}\|_{PW^{p'}_{\pi}}\}_l$ is
uniformly minimal in $PW^{p'}_{\pi}$ but not unconditional.

Still, it can be observed that $\Gamma$ is uniformly separated
in the euclidean distance and hence by the classical 
Plancherel-P\'olya inequality we have for every 
$f\in PW^p_{\pi}$
\bea\label{PPineq}
 \sum_{k}|f(\gamma_k)|^p\le C\|f\|^p_p,
\eea
so that the restriction operator $f\lmto f|\Gamma$ is continuous
from $PW^p_{\pi}$ to $l^p$ (onto when $\Gamma$ is interpolating),
in other words the measure $\sum_{\gamma\in \Gamma}\delta_{\gamma}$
is $PW^p_{\pi}$-Carleson.


More can be said.
The following result is nothing but a re-interpretation
of \cite{LS}.

\begin{proposition}\label{prop2.4}
Let $1<p\le 2$. Then for every $1<s<p$ there exists a sequence
$\Gamma$ that is interpolating for $PW^p_{\pi}$ without
being interpolating for $PW^s_{\pi}$.
\end{proposition}

So, in the scale of Paley-Wiener spaces --- which represents
a subclass of backward shift invariant subspaces --- an
interpolating sequence is not necessarily interpolating in 
an arbitrary bigger space, and so {\it a fortiori} a dual bounded
sequence for a given $p$ is not necessarily interpolating for
a bigger space $K^s_I$, $s<p$. This should motivate why in our
main result discussed in the next section 
we increase the space in two directions to get
interpolation from dual boundedness: we increase the space
by adding factors to the defining inner function {\it and} by
decreasing $p$.

Again we translate the result to the language of unconditionality.
The sequence constructed in this proposition is again a real sequence
which is uniformly separated in the euclidean metric so that 
\eqref{PPineq} holds for $p$ and $s$ and hence the measure
$\sum_{k\in \Z}\delta_{\gamma_k}$ is a Carleson measure. 
This implies that if $\Gamma$ is interpolation for $PW^p_{\pi}$
then  we do not
only have $PW^p_{\pi}|\Gamma\supset l^p$ (recall that the reproducing
kernel is given by the $\sinc$-function in $\gamma_k\in \R$ 
the norm of which is comparable to a constant)
but $PW^p_{\pi}|\Gamma = l^p$. By Corollary \ref{cor2.3new} this
means that $(k_{\gamma}/\|k_{\gamma}\|_{p'})_{\gamma\in\Gamma})$ 
is unconditional in $PW^{p'}_{\pi}$. Clearly, since $\Gamma$ is
not interpolating for $PW^{s}_{\pi}$, the sequence
$(k_{\gamma}/\|k_{\gamma}\|_{s'})_{\gamma\in\Gamma}$ 
cannot be unconditional in $PW^{s'}_{\pi}$. 
We recapitulate these observations in
the following result.

\begin{corollary}
Let $1<p\le 2$. Then for every $1<s<p$ there exists a sequence
$\Gamma$ such that $(k_{\gamma}/\|k_{\gamma}\|_{p'})_{\gamma\in\Gamma})$ 
is unconditional in $PW^{p'}_{\pi}$ and 
$(k_{\gamma}/\|k_{\gamma}\|_{s'})_{\gamma\in\Gamma}$ 
is not unconditional for $PW^{s'}_{\pi}$.
\end{corollary}

Recall that $k_x$, the reproding kernel in $PW^2_{\tau}$ is
given by a $\sinc$-function the norm of which is comparable to
a constant when $x\in \R$.

It can be noted that $s'>p'$ so that $PW^{s'}_{\pi}$ is a smaller space
than $PW^{p'}_{\pi}$.



\begin{proof}[Proof of Proposition \ref{prop2.4}]
Since $1<p\le 2$ we have $p_0:=\max(p,p')=p'$ (recall
$1/p+1/p'=1$). In contrast to the above example where we have
'spread out' slightly the integers (by adding a constant
to the positive integers and subtracting the same constant
from the negative integers) to obtain a dual bounded
sequence which is not interpolating ($p\ge 2$) we will now
narrow the integers: let
$\delta_k=-\sign(k)/2s'$. We have in particular
$s_0=\max(s,s')=s'>p'$. Define $\Gamma=(\gamma_k)_{k\in\Z}$
by $\gamma_k=k+\delta_k$,
$k\in \Z\setminus\{0\}$, $\gamma_0=0$.
Then 
as the example in \cite[Theorem 2]{LS}, the sequence $\Gamma$
is not interpolating for $PW^s_{\pi}$.
On the other hand, since $|\delta_k|=1/2s'<1/2p'$ we deduce
from the sufficiency part of 
\cite[Theorem 2]{LS} that $\Gamma$ is complete interpolating for
$PW^p_{\pi}$.
\end{proof}


\begin{remark}
We have mentioned the translations $\Phi_a$, $a\in \R$. These
allow to translate the above example $\Gamma$ to any line
parallel to the real axis: $\Phi_a\Gamma$. By the properties of
$\Phi_a$, we keep the properties of uniform minimality and
(non)-interpolation.
\end{remark}


We now discuss the effect of increasing the size of the
space in the Paley-Wiener case ``in the direction of the
inner function''. More precisely we will consider the
situation when we replace $I$ by $I^{1+\eps}$ on the
$K^p_I$-side, which means on the
Paley-Wiener side  that we replace the type 
$\pi$ by $\pi (1+\eps)=:\pi+\eta$ for some $\eta>0$.
And for $p=2$, on the Fourier side this means that we
replace $[-\pi,\pi]$ by $[-(\pi+\eta),\pi+\eta]$.

We will use \cite[Theorem 2.4]{Se95} to prove the following result.

%
%
%
%
%
%
%
%
%

\begin{proposition}\label{prop2.5}
Let $\Gamma=\{\gamma_k\}_{k\in\Z}$ be defined by
$\gamma_0=0$, $\gamma_k=k+\sign(k)/4$. Then $(k_{\gamma})_{\gamma\in
  \Gamma}$ is uniformly minimal and not 
unconditional in $PW^2_{\pi}$, and for every $\eta>0$, $\Gamma$
is an unconditional sequence in $PW^2_{\pi+\eta}$.
\end{proposition}

\begin{proof}[Proof of Proposition \ref{prop2.5}]
The first part of the claim 
is established by Theorem \ref{thm3.1}.

We use \cite[Theorem 2.4]{Se95} for interpolation in the bigger space.
Seip's theorem furnishes a sufficient density
condition for unconditional sequences in Paley-Wiener spaces when $p=2$
which makes this proof very easy. Recall that $n^+(r)$ denotes
the largest number of points from a sequence of real numbers 
$\Lambda$ to be found in 
an interval of length $r$. The upper uniform density is then
defined as
\beqa
 D^+(\Lambda):=\lim_{r\to\infty}\frac{n^+(r)}{r}
\eeqa
(the limit exists by standard arguments on subadditivity of $n^+(r)$).
\cite[Theorem 2.4]{Se95} states that when a 
sequence $\Lambda$, which is uniformly separated in the euclidean
distance, satisfies $D^+(\Lambda)<\frac{\D \tau}{\D 2\pi}$,
then $(k_{\lambda}/\|k_{\lambda}\|_{PW^2_{\tau}})_{\lambda\in
\Lambda}$ is an unconditional sequence in  $PW^2_{\tau}$
(strictly speaking Seip's theorem yields the unconditionality 
for exponentials in $L^2([-\tau,\tau])$, but via the Fourier transform
this is of course the same as for reproducing kernels). 
Our sequence $\Gamma$ clearly satisfies $D^+(\Gamma)=1$, and hence
whenever $\tau>2\pi$, then $\Gamma$ is interpolating in $PW^2_{\tau}$.
\end{proof}

The proposition can also be shown by appealing to 
\cite[Theorem 3]{SchS2} which gives a kind of uniform 
non-uniqueness condition as sufficient condition for interpolation
in Paley-Wiener spaces. It can in fact be shown using a perturbation 
result by Redheffer that the
weak limits (in the sense of Beurling) of our sequence $\Gamma$
have the same completeness radius (in the sense of Beurling-Malliavin)
as $\Gamma$, i.e. $\pi$. So increasing the size of the interval makes
these weak limits non-uniqueness in the bigger space (this is the
most difficult condition of Schuster and Seip's result to be checked;
concerning the other conditions appearing in their theorem, i.e.\ 
uniform separation and the two-sided Carleson condition,
these are immediate).

\begin{question*}
A natural question arises in the context of these results. 
Is it possible that the sequence $\Gamma$ of Proposition \ref{prop2.5}
--- which is dual bounded but not interpolating in $PW^2_{\pi}$ ---
is interpolating in $PW^p_{\pi}$ for some 
$p=2-\eps$ (or $p$ in some intervalle $(2-\eps,2)$)
for suitable small $\eps$? 
\end{question*}

So this time we increase the 
size of the space in the direction $p$. Proposition \ref{prop2.4}
indicates that $\eps$ cannot be chosen arbitrarily big.
This proposition also motivates 
another important remark. A sufficient condition for interpolation
in terms of a suitable density and depending on the value of $p$, as
encountered e.g.\ in the context of  Bergman 
spaces where a sequence satisfying the criticial density
is automatically interpolating in the bigger spaces, seems
not expectable. This makes 
the question very delicate (note that the sequence $\Gamma$ of
Proposition \ref{prop2.5} has the critical density for $PW^2_{\pi}$).\\


\section{The main result}\label{S3}

Let $I$ be an inner function, i.e.\ a function analytic on $\DD$,
bounded by $1$, and such that $|I(\zeta)|=1$ for a.e.\ $\zeta\in\T$.
Such a function is called one-component when there exists an
$\eps\in (0,1)$ such that $L(I,\eps)=\{z\in\DD:|I(z)|<\eps\}$
is connected. Simple examples of such functions are for 
example $I(z)=\exp((z+1)/(z-1))$ or Blaschke products with
zeros not ``too far'' such as $B_{\Lambda}$ associated with
the interpolating sequence $\Lambda=\{1-1/2^n\}_n$.
One-component inner functions appear for example in the context
of embeddings for star invariant subspaces. For example, Treil
and Volberg \cite{TV} discuss the embedding $K^p_I\subset L^p(\mu)$
when $I$ is one-component.

The following result will be of interest for us

\begin{theorem*}[\cite{Al1}]
If $I$ is a one-component inner function and $1<p\le \infty$, then
\bea\label{Aleks}
 C_1(I,p)\left(\frac{1-|I(a)|^2}{1-|a|^2}\right)^{1-1/p}
 \le \left\|\frac{1-\overline{I(a)}I(z)}{1-\overline{a}z}\right\|_p
 \le  C_2(I,p)\left(\frac{1-|I(a)|^2}{1-|a|^2}\right)^{1-1/p}
\eea
for all $a\in\DD$.
\end{theorem*} 

We will now discuss the principal results
that lead to Theorem \ref{thm1}.


For a sequence $S$  of points in ${\mathbb{D}},$ 
we introduce the related sequence 
$\{\epsilon _{a}\}_{a\in S}$ of independent 
Bernoulli variables.

We now increase $K^p_I$, when $p$ is fixed, which means that we multiply
a factor to the inner function  $I$.
More precisely let $J=IE$ where $E$ is another inner function. 
Recall that $\overline{K^p_I+IK^p_E}=K^p_J$ (which gives an idea
on the increase of the space; note that this identity can also be
derived from a more general one in de Branges-Rovnyak spaces).

We first discuss 
when dual boundedness for 
$p>1$ implies interpolation for $q=1$.

\begin{lemma} \label{casep=1}
Let $S\subset {\mathbb{D}}$ be 
dual bounded in $K^p_I$,
$p>1$, 
and let $E$ be another inner function.
If 
\begin{equation} 
\displaystyle \Vert{\displaystyle k_{a}^{J}}\Vert _{\infty 
}\simeq \frac{\displaystyle \displaystyle \Vert{\displaystyle k_{a}^{I}}\Vert 
_{p'}\displaystyle \Vert{\displaystyle k_{a}^{E}}\Vert _{2}^{2}}{\displaystyle 
\displaystyle \Vert{\displaystyle k_{a}^{E}}\Vert _{p'}},\label{lemmeP1}
\end{equation} 
then $S$ is interpolating in $K_{J}^{1}$ with $J=IE.$\ \par
\end{lemma}

\begin{proof}
Let first 
$c_a=\frac{\displaystyle 
\displaystyle \Vert{\displaystyle k_{a}^{E}}\Vert _{p'}\displaystyle 
\Vert{\displaystyle k_{a}^{J}}\Vert _{\infty }}
{\displaystyle \displaystyle \Vert{\displaystyle k_{a}^{I}}\Vert 
_{p'}\displaystyle \Vert{\displaystyle k_{a}^{E}}\Vert _{2}^{2} }$ which
is comparable to a uniform constant.

Since $S$ is dual bounded in $K^p_I$, the sequence
$(k^I_a/\|k^I_a\|_{p'})_{a\in S}$ is uniformly minimal, so that
there exists a dual sequence $(\rho_{p,a})_{a\in S}$ in 
$K^p_I$: 
$\langle \rho_{p,a}, k^I_{p',b}\rangle=\delta_{ab}$, i.e.\ 
$\rho_{p,a}(b)=\delta_{ab}\|k^I_{b}\|_{p'}$, and 
$\sup_{a\in S}\|\rho_{p,a}\|_p<\infty$.
As in \cite{Amar} the idea is now to take
\beqa
  \forall \lambda \in \ell ^{1},\ 
 T(\lambda ):= \sum_{a\in S}^{}{\lambda _{a}
 c_a\rho _{p,a}\frac{ k_{a}^{E}}{\left\Vert{ k_{a}^{E}}\right\Vert _{p'}}.}
\eeqa
The sum defining $T$ converges clearly under the assumption of 
the theorem since $\lambda$ is summable.
Also $k_a^E(a)=\|k^E_a\|_2^2$, and hence
\beqa
 T(\lambda )(a)=\lambda _{a}c_a\rho _{a,p}(a)\frac{k_{a}^{E}(a)}
  {\left\Vert{k_{a}^{E}}\right\Vert _{p'}}
  =\lambda _{a}c_a\frac{\Vert{k_{a}^{I}}\Vert _{p'}
   \Vert{k_{a}^{E}}\Vert_{2}^{2}}{\Vert{k_{a}^{E}}
   \Vert _{p'}}= \lambda_a \|k_a^J\|_{\infty}.
\eeqa
So, by equation \eqref{lemmeP1}, $S$ is interpolating in 
$K_{J}^{1}$.
\end{proof}

We shall now discuss the general situation.

\begin{lemma} \label{algUniExt411}
Suppose that $I$ and $E$ are one-component inner functions.
Let $S\subset {\mathbb{D}}$ be a 
dual bounded sequence in $K^p_I$; 
let $1\leq s<p$ and $q$ be such that $\displaystyle 
\frac{\displaystyle 1}{\displaystyle s}=\frac{\displaystyle 1}{\displaystyle 
p}+\frac{\displaystyle 1}{\displaystyle q};$ 
suppose that the following conditions are satisfied.
\begin{itemize}
\item[(i)] 
$\Vert{k_{a}^{J}}\Vert _{s'}\simeq
\frac{\displaystyle\Vert{k_{a}^{E}}\Vert 
_{s'}\Vert{k_{a}^{I}}\Vert _{p'}}{\displaystyle\Vert{k_{a}^{E}}\Vert 
_{p'}}$;
\item[(ii)] $\displaystyle \forall \lambda \in \ell ^{p}(S),\ {\mathbb{E}}\displaystyle \left[{\displaystyle 
\displaystyle \left\Vert{\displaystyle \displaystyle \sum_{a\in S}^{}{\lambda 
_{a}\epsilon _{a}\rho _{p,a}}}\right\Vert _{p}^{p}}\right] \lesssim\displaystyle 
\left\Vert{\displaystyle \lambda }\right\Vert _{\displaystyle \ell
  ^{p}}^{p}$
\item[(iii)] if $q>2,\ S$ is weakly $q$-Carleson in $K_{E}^{q},$ 
\end{itemize}
Then $S$ is $\displaystyle K_{J}^{s}$ interpolating and moreover 
there exists a bounded linear interpolation operator 
$T:l^s(S)\lra K^s_J$, $T(\nu)(a)=\nu_a  \|k_a^J\|_{s'}$.
\end{lemma}

Observe that we do not need to require the 
Carleson condition on $S$ when $q\leq 2.$

\begin{remark}\label{rem4.1}
Before proving the result, we discuss 
some special cases where the condition (i) is satisfied.
Recall from  \eqref{Aleks}
that for an arbitrary inner one-component function $\Theta$ we
have 
\bea\label{equivrel1}
 \displaystyle \displaystyle \Vert{\displaystyle k_{a}^{\Theta}}
  \Vert _{s'}
 \simeq \displaystyle \left({\displaystyle 
  \frac{\displaystyle 1-\displaystyle \left\vert{\displaystyle
        \Theta(a)} \right\vert ^{2}}{\displaystyle 1-\displaystyle 
   \left\vert{\displaystyle a}\right\vert ^{2}}}\right) ^{1/s}.
\eea
Hence when $I,E$ are one-component,
we get
\bea\label{equivrel2}
 \displaystyle \frac{\displaystyle \Vert{\displaystyle k_{a}^{E}}\Vert 
   _{s'}\displaystyle\Vert{\displaystyle k_{a}^{I}}\Vert _{p'}}
   {\displaystyle \Vert{\displaystyle k_{a}^{E}}\Vert _{p'}}
 &\simeq& \frac{\displaystyle \left({\displaystyle 1-\displaystyle 
  \vert{\displaystyle E(a)}\vert ^{2}}\right) ^{1/s}\displaystyle 
  \left({\displaystyle 1-\vert{\displaystyle I(a)} \vert ^{2}}\right) 
   ^{1/p}}{\displaystyle \left({\displaystyle 
    1-\displaystyle \vert{\displaystyle E(a)}\vert ^{2}}\right)^{1/p}
   \left({\displaystyle 1-\displaystyle \left\vert{\displaystyle a}
   \right\vert ^{2}}\right) ^{1/s}}\nn\\
 &=&\displaystyle \frac{\left({\displaystyle 1-\displaystyle \left
  \vert{\displaystyle E(a)}\right\vert ^{2}}\right) ^{1/q}
  \displaystyle \left({\displaystyle 1-\displaystyle 
  \left\vert{\displaystyle I(a)}\right\vert ^{2}}\right) ^{1/p}}
  {\displaystyle \left({\displaystyle 1-\displaystyle \left
  \vert{\displaystyle a}\right\vert ^{2}}\right) ^{1/s}}.
\eea
From this we can deduce that (i) holds in
the following cases.
\begin{itemize}
\item[(1)]
Suppose $E,I$ are one-component and
$\displaystyle \sup _{a\in S}\displaystyle \left\vert{\displaystyle E(a)}\right\vert 
\leq \eta <1$ and $\displaystyle \sup _{a\in S}\displaystyle \left\vert{\displaystyle I(a)}\right\vert 
\leq \eta <1$. 
Suppose also that $J=IE$ is one-component (it is not clear whether
this follows from $I$ and $E$ being one-component). Clearly
$\sup_{a\in S} |J(a)|<1$, and (i) follows.
\item[(2)]
$E=I$ and $I$ is one-component, then $J=I^{2}$ (note that it is clear
that when $L(I,\eps)$ is connected then so is $L(I^2,\eps^2)$); in 
this case we do not need the $\sup$-condition, since
\beqa
  \displaystyle \displaystyle 
  \left({\displaystyle 1-\displaystyle \left\vert{\displaystyle 
   I(a)}\right\vert ^{4}}\right) ^{1/s}
 \simeq \displaystyle \left({\displaystyle 1-\displaystyle 
   \left\vert{\displaystyle I(a)}\right\vert ^{2}}\right) ^{1/q}
  \displaystyle \left({\displaystyle 1-\displaystyle \left\vert{\displaystyle 
   I(a)}\right\vert ^{2}}\right) ^{1/p},
\eeqa
which by \eqref{equivrel1} and \eqref{equivrel2} yields (i);
\item[(3)]
$I$ singular and $\forall \alpha >0,\ E=I^{\alpha
}$ which implies $J=I^{1+\alpha }.$
\end{itemize}
\end{remark}

\begin{remark} If $p=1$ then dual boundedness of $S$ 
in $K_{I}^{1}$ implies that $S$ interpolating in $K_{I}^{1}$
(take the interpolation operator constructed in the proof
of Lemma \ref{casep=1}).
\end{remark}

\begin{proof}[Proof of the Lemma]
In view of Lemma \ref{casep=1} we can suppose $1<s<p$. 

In order to prove the lemma we will construct a function
$f$ interpolating a sequence $\nu\in l^s$ weighted by the norm
of the reproducing kernels.
To do this,
we will consider finitely supported
sequences $\nu$, say with only the first $N$ components possibly
different from zero, and check that the constants do not depend
on $N\in \N$.
So, for $1<s<p$ and $\nu \in \ell _{N}^{s}$ we shall build a function 
$h\in K_{J}^{s}$ such that:
\beqa
 \forall j=0,...,N-1,\ h(a_{j})=\nu _{j}\Vert{k^J_{a_{j}}}\Vert _{s'}
 \text{ and } \Vert{h}\Vert _{K_{J}^{s}}\leq C\Vert{\nu }\Vert 
 _{\ell _{N}^{s}}.
\eeqa
where the constant $C$ is independent of $N$. The conclusion follows 
from a normal families argument (see also \cite{Am07}).

We choose $q$ such that $\displaystyle \frac{\displaystyle 1}{\displaystyle s}=\frac{\displaystyle 1}{\displaystyle 
p}+\frac{\displaystyle 1}{\displaystyle q};$ then $q\in ]p',\infty [$ 
with $p'$ the conjugate exponent of $p$ and we set $\nu _{j}=\lambda _{j}\mu _{j}$ 
with $\displaystyle \mu _{j}:=\displaystyle \left\vert{\displaystyle \nu _{j}}\right\vert 
^{s/q}\in \ell ^{q},\ \lambda _{j}:=\frac{\displaystyle \nu _{j}}{\displaystyle 
\displaystyle \left\vert{\displaystyle \nu _{j}}\right\vert }\displaystyle 
\left\vert{\displaystyle \nu _{j}}\right\vert ^{s/p}\in \ell ^{p}$ so that
$\left\Vert{\nu }\right\Vert _{s}=\left\Vert{\lambda }\right\Vert _{p}\left\Vert{\mu 
}\right\Vert _{q}.$

Let now
\beqa
 \displaystyle c_{a}:=
\frac{\displaystyle \Vert{\displaystyle k_{a}^{E}}\Vert _{q}
 \displaystyle \Vert{\displaystyle k_{a}^{J}}\Vert _{s'}}
  {\displaystyle  \Vert{\displaystyle k_{a}^{I}}\Vert
    _{p'}k_{a}^{E}(a)}.
\eeqa
By (i), we have
\beqa
 c_a\simeq \frac{\displaystyle \Vert{\displaystyle k_{a}^{E}}\Vert _{q}
 \displaystyle \Vert{\displaystyle k_{a}^{E}}\Vert _{s'}}
  {\displaystyle  \Vert{\displaystyle k_{a}^{E}}\Vert
    _{p'}k_{a}^{E}(a)}=
 \frac{\displaystyle \Vert{\displaystyle k_{a}^{E}}\Vert _{q}
 \displaystyle \Vert{\displaystyle k_{a}^{E}}\Vert _{s'}}
  {\displaystyle  \Vert{\displaystyle k_{a}^{E}}\Vert
    _{p'}\|k_{a}^{E}\|_2^2}.
\eeqa
Since $E$ is one-component 
we have \eqref{Aleks}, i.e.\
$\displaystyle \displaystyle \Vert{\displaystyle k_{a}^{E}}\Vert _{r}
 \simeq \displaystyle \left({\displaystyle \frac{\displaystyle 1-\displaystyle 
\left\vert{\displaystyle E(a)}\right\vert ^{2}}{\displaystyle 1-\displaystyle 
\left\vert{\displaystyle a}\right\vert ^{2}}}\right) ^{1/r'},
$
where $\frac{1}{r}+\frac{1}{r'}=1$. 
Clearly $1/q'+1/s-1/p+2\times 1/2=1/q'+1/q-1=0$, and
hence  $c_{a}\simeq C,$ 
the constant being independent of $a\in S.$

Next set $h(z)
:=T(\nu)(z):=\sum_{a\in S}^{}{\nu _{a}c_{a}\rho _{a}k_{q,a}^{E}}$. 
Then, because $\rho _{a}(b)=\delta _{ab}\Vert{k_{a}^{I}}\Vert _{p'}:$
\beqa
 \displaystyle \forall a\in S,\ h(a)=\nu _{a}c_{a}
  \displaystyle \Vert{\displaystyle 
  k_{a}^{I}}\Vert _{p'}k_{q,a}^{E}(a).
\eeqa

Recall that $\displaystyle k_{q,a}^{E}(a)={\displaystyle k_{a}^{E}(a)}/
{\displaystyle \displaystyle \Vert{\displaystyle k_{a}^{E}}\Vert
  _{q}}$.
Hence
\beqa
 h(a)=\nu_a c_{a}\|k_a^I\|_{p'}k_{q,a}^{E}(a)
 =\nu_a  \times\frac{\displaystyle \Vert{\displaystyle 
  k_{a}^{E}}\Vert _{q}\displaystyle \Vert{\displaystyle
  k_{a}^{J}}\Vert_{s'}}{\displaystyle \Vert{\displaystyle k_{a}^{I}}\Vert 
  _{p'}k_{a}^{E}(a)}\times
 \|k_a^I\|_{p'}\times \frac{\displaystyle k_{a}^{E}(a)}{\displaystyle 
  \Vert{\displaystyle k_{a}^{E}}\Vert _{q}}
 =\nu _{a}\displaystyle
 \Vert{\displaystyle k_{a}^{J}}\Vert_{s'}
\eeqa
and $h$ satisfies the interpolation condition.

Let us now come to the estimate of the $K_{J}^{s}$ norm of $h.$

Set
\beqa
 f(\epsilon ,z):=\sum_{a\in S}^{}{\lambda _{a}c_{a}\epsilon _{a}\rho
   _{a}(z)},
 \quad \text{and}\quad g(\epsilon ,z):=\sum_{a\in S}^{}{\mu _{a}\epsilon 
_{a}k_{q,a}^{E}(z)}.
\eeqa
Then $\displaystyle h(z)={\mathbb{E}}(f(\epsilon ,z)g(\epsilon ,z))$ 
because $\displaystyle {\mathbb{E}}(\epsilon _{j}\epsilon _{k})=\delta
_{jk}.$

So we get
\beqa
 \displaystyle \vert{\displaystyle h(z)}\vert ^{s}
 =\vert{\displaystyle {\mathbb{E}}(fg)}\vert ^{s}
 \leq ({\mathbb{E}}(\displaystyle 
   \vert{\displaystyle fg}\vert ))^{s}
 \leq {\mathbb{E}}(\displaystyle 
  \vert{\displaystyle fg} \vert ^{s}),
\eeqa
and hence
\beqa
  \Vert{\displaystyle h}\Vert _{s}
 =\displaystyle 
  \left({ \int_{{\mathbb{T}}}^{}{\displaystyle
  \vert{\displaystyle h(z)}\vert ^{s}\,d\sigma (z)}}\right) ^{1/s}
 \leq \displaystyle \left({\displaystyle \displaystyle 
 \int_{{\mathbb{T}}}^{}{{\mathbb{E}}(\displaystyle 
 \vert{\displaystyle fg}\vert ^{s})\,d\sigma (z)}}\right) ^{1/s}.
\eeqa

By H{\"o}lder's inequality, we get
\begin{equation} 
\displaystyle \int_{{\mathbb{T}}}^{}{{\mathbb{E}}(\displaystyle \left\vert{\displaystyle 
fg}\right\vert ^{s})\,d\sigma (z)}={\mathbb{E}}\displaystyle \left[{\displaystyle 
\displaystyle \int_{{\mathbb{T}}}^{}{\displaystyle \left\vert{\displaystyle 
fg}\right\vert ^{s}\,d\sigma (z)}}\right] \leq \displaystyle \left({\displaystyle 
{\mathbb{E}}\displaystyle \left[{\displaystyle \displaystyle \int_{{\mathbb{T}}}^{}{\displaystyle 
\left\vert{\displaystyle f}\right\vert ^{p}\,d\sigma }}\right] }\right) 
^{s/p}\displaystyle \left({\displaystyle {\mathbb{E}}\displaystyle \left[{\displaystyle 
\displaystyle \int_{{\mathbb{T}}}^{}{\displaystyle \left\vert{\displaystyle 
g}\right\vert ^{q}\,d\sigma }}\right] }\right) ^{s/q}.\label{algUniExt410}
\end{equation}

Now for $ a\in S$, set $ \tilde \lambda _{a}:=c_{a}\lambda _{a}$.
Then $\Vert{\tilde \lambda }\Vert _{p}\leq C\Vert{\lambda }\Vert _{p}$ and 
the first factor in \eqref{algUniExt410} is controlled by (ii) of the
hypotheses of the Lemma:
\begin{equation} 
{\mathbb{E}}\displaystyle \left[{\displaystyle \displaystyle \int_{{\mathbb{T}}}^{}{\displaystyle 
\left\vert{\displaystyle f}\right\vert ^{p}\,d\sigma }}\right] ={\mathbb{E}}\displaystyle 
\left[{\displaystyle \displaystyle \left\Vert{\displaystyle \displaystyle 
\sum_{a\in S}^{}{\lambda _{a}c_{a}\epsilon _{a}\rho _{p,a}}}\right\Vert 
_{p}^{p}}\right] \lesssim \displaystyle \left\Vert{\displaystyle \tilde 
\lambda }\right\Vert _{p}^{p}\lesssim \displaystyle \left\Vert{\displaystyle 
\lambda }\right\Vert _{\displaystyle \ell ^{p}}^{p},\label{algUniExt48}
\end{equation} 
and the constants appearing here do not depend on $N$.

Consider the second factor in \eqref{algUniExt410}.
Fubini's theorem  gives: 
\beqa
 \displaystyle {\mathbb{E}}\displaystyle \left[{\displaystyle
     \displaystyle \int_{{\mathbb{T}}}^{}{\displaystyle 
   \left\vert{\displaystyle g}\right\vert ^{q}\,d\sigma }}\right] 
 =\displaystyle \int_{{\mathbb{T}}}^{}{{\mathbb{E}}\displaystyle 
  \left[{\displaystyle \left\vert{\displaystyle g}\right\vert
      ^{q}}\right] 
   \,d\sigma }.
\eeqa

We apply Khinchin's inequalities to 
$\displaystyle {\mathbb{E}}\displaystyle 
\left[{ \displaystyle \left\vert{\displaystyle g}\right\vert
    ^{q}}\right] $:
\beqa
 \displaystyle {\mathbb{E}}\displaystyle \left[{\displaystyle
     \displaystyle \left\vert{\displaystyle g}\right\vert ^{q}}\right] 
 \simeq \displaystyle \left({\displaystyle \sum_{a\in S}^{}
  {\displaystyle \left\vert{\displaystyle \mu _{a}}\right\vert ^{2}
  \displaystyle \left\vert{\displaystyle k_{q,a}^{E}}\right\vert
  ^{2}}}\right) ^{q/2}.
\eeqa
If $q>2$, then $S$ weakly $q$-Carleson implies
\begin{equation} 
 \displaystyle \int_{{\mathbb{T}}}^{}{{\mathbb{E}}\displaystyle 
  \left[{\displaystyle \displaystyle 
  \left\vert{\displaystyle g}\right\vert ^{q}}\right] \,d\sigma }
 \lesssim \displaystyle \int_{{\mathbb{T}}}^{}{\displaystyle 
  \left({\displaystyle \displaystyle \sum_{a\in S}^{}{\displaystyle 
  \left\vert{\displaystyle \mu _{a}}\right\vert ^{2}\displaystyle\left 
  \vert{\displaystyle k_{q,a}^{E}}\right\vert ^{2}}}\right)^{q/2}\,d\sigma }
 \lesssim \displaystyle \left\Vert{\displaystyle 
   \mu }\right\Vert _{\displaystyle \ell ^{q}}^{q},\label{algUniExt49}
\end{equation} 
where, again, the constants do not depend on $N$.

If $q\leq 2$ then $\displaystyle \displaystyle \left({\displaystyle 
\displaystyle \sum_{a\in S}^{}{\displaystyle 
\left\vert{\displaystyle \mu _{a}}\right\vert ^{2}\displaystyle 
\left\vert{\displaystyle 
k_{q,a}^{E}}\right\vert ^{2}}}\right) ^{q/2}\leq \displaystyle \sum_{a\in 
S}^{}{\displaystyle \left\vert{\displaystyle \mu _{a}}\right\vert ^{q}\displaystyle 
\left\vert{\displaystyle k_{q,a}^{E}}\right\vert ^{q}},$ and integrating 
over $\T$ we get:

\begin{equation} 
\displaystyle \int_{{\mathbb{T}}}^{}{{\mathbb{E}}\displaystyle \left[{\displaystyle 
\displaystyle \left\vert{\displaystyle g}\right\vert ^{q}}\right] \,d\sigma 
}\leq \displaystyle \int_{{\mathbb{T}}}^{}{\displaystyle \left({\displaystyle 
\displaystyle \sum_{a\in S}^{}{\displaystyle \left\vert{\displaystyle 
\mu _{a}}\right\vert ^{q}\displaystyle \left\vert{\displaystyle k_{q,a}^{E}}\right\vert 
^{q}}}\right) \,d\sigma }\leq \displaystyle \sum_{a\in S}^{}{\displaystyle 
\left\vert{\displaystyle \mu _{a}}\right\vert ^{q}\displaystyle \int_{{\mathbb{T}}}^{}{\displaystyle 
\left\vert{\displaystyle k_{q,a}^{E}}\right\vert ^{q}\,d\sigma }}=\displaystyle 
\left\Vert{\displaystyle \mu }\right\Vert _{\displaystyle \ell ^{q}}.\label{algUniExt152}
\end{equation} 

So putting \eqref{algUniExt48} and \eqref{algUniExt49} or \eqref{algUniExt152} 
in \eqref{algUniExt410} we get that $S$ is an interpolating
sequence  for $K^s_J$. Clearly the operator $T$ is a bounded linear
interpolation operator. 
\end{proof}

We are now in a position to prove the main result of this paper.

\begin{theorem}\label{thm3.2} 
Let $1<p\leq 2,\ 1\leq s<p$ and $q$  such that $\displaystyle 
\frac{\displaystyle 1}{\displaystyle s}=\frac{\displaystyle 1}{\displaystyle 
p}+\frac{\displaystyle 1}{\displaystyle q}.$ Suppose that 
\begin{itemize}
\item[(i)] the dual sequence $\{\rho _{p,a}\}_{a\in S}$ 
exists and is norm bounded in $K_{I}^{p}$, 
\item[(ii)] $\Vert{k_{a}^{J}}\Vert _{s'}\simeq 
 \frac{\displaystyle \Vert{k_{a}^{E}}\Vert _{s'}
  \Vert{\displaystyle k_{a}^{I}}\Vert _{p'}}
  {\displaystyle \Vert{k_{a}^{E}}\Vert _{p'}}$ and
\item[(iii)] $S$ is weakly $q$-Carleson in $K_{E}^{q}.$
\end{itemize}
Then $S$ is $K_{J}^{s}$-interpolating 
and there exists a bounded linear interpolation operator.
\end{theorem}

Before discussing special cases we mention a first consequence
(using Proposition \ref{prop2.2} and Fact \ref{fact1})
for the case of unconditionality.

\begin{corollary}\label{cor4.7new}
Suppose the conditions of the preceding theorem fulfilled.
Assume moreover that $J$ is one-component and and that we
have condition (iii) of Fact \ref{fact1}: the measure
$\nu=\sum_{a\in S}\frac{\D 1-|a|^2}{\D 1-|J(a)|^2}\delta_a$
satisfies
\bea
  |\mu|(S(\zeta,r))\le C r
\eea
for every Carleson window $S(\zeta=e^{it},h)$ meeting
the level set $L(J,1/2)$.
Then $(k_a^J/\|k_a^J\|_{s'})_{a\in S}$ is an unconditional
sequence in $K^{s'}_J$.
\end{corollary}

As a corollary we obtain 
the first part of Theorem \ref{thm1}.

\begin{corollary}\label{cor4.7}
Let $1<p\le 2$.
Let $I$ be a one-component singular inner function and
$S\subset \DD$. Suppose
that $\sup_{a\in S}|I(a)|<1$. If
$(k_{a}^I/\|k_{a}^I\|_{p'})_{a\in S}$
is uniformly minimal in $K^{p'}_I$, where $1/p+1/p'=1$ 
then for every $\eps>0$ and
for every $1\le s<p$, 
$S$ is an interpolating sequence in $K^s_{I^{1+\eps}}$.
\end{corollary}

\begin{proof}[Proof of Corollary \ref{cor4.7}]
Condition (ii) of the theorem follows from the
case (3) of Remark \ref{rem4.1}. 
The condition (i) of the theorem is fulfilled by the fact
that $(k_{a}^I/\|k_{a}^I\|_{p'})_{a\in S}$
is uniformly minimal in $K^{p'}_I$. Let $(\rho_{p,a})_{a\in S}$ be
the corresponding dual family in $K^p_I$.
It remains to check the weak $q$-Carleson
condition. In fact more is true:
Since $I$ is one-component and inner with
$\sup_{a\in S}|I(a)|<1$, we have for every $a\in S$, $1<r<\infty$
\beqa
 \|k_{a}^I\|_r\simeq
 \left(\frac{1-|I(a)|^2}{1-|a|^2}\right)^{1-1/r}
 \simeq \left(\frac{1}{1-|a|^2}\right)^{1-1/r}
 \simeq\|k_{a}\|_r.
\eeqa
Hence, up to some constants $c_a$, $a\in S$, whose moduli are uniformly
bounded above and below we get
\beqa
 \delta_{ab} &=& \langle \rho_{p,a},k_{b}^I/\|k_{b}^I\|_p\rangle
 =c_{a} \langle \rho_{p,a},k_{b}^I/\|k_{b}\|_p\rangle
 =c_{a} \langle \rho_{p,a},P_I(k_{b}/\|k_{b}\|_p)\rangle\\
 &=&c_{a} \langle  P_I\rho_{p,a},k_{b}/\|k_{b}\|_p\rangle\\
 &=&c_{a} \langle  \rho_{p,a},k_{b}/\|k_{b}\|_p\rangle.
\eeqa
Hence $(k_{a}/\|k_{a}\|_{p'})_{a\in S}$ is a uniform minimal
sequence in $H^p$ which by the interpolation results
is 
equivalent to $\Lambda\in (C)$. 
(We could also have shown this by using directly
\eqref{unifmin}.) In particular,
$(k_{a}/\|k_{a}\|_{p'})_{a\in S}$ is an unconditional sequence
in any $H^r$, $1<r<\infty$.

From this we can deduce that $S$ is even $r$-Carleson for
any $1<r<\infty$: indeed,
let $(\mu_a)_{a\in S}\in l^r$, then
\bea\label{scarl}
 \left\|\sum_{a\in S}\mu_a k_{a,r}^I\right\|_r^r
 &=&\left\|P_I\sum_{a\in S}\mu_a
   \frac{k_{a}}{\|k_{a}^I\|_r}\right\|_r^r
 \le c \left\|\sum_{a\in S}\mu_a
   \frac{k_{a}}{\|k_{a}^I\|_r}\right\|_r^r
 =c \left\|\sum_{a\in S}\mu_a\frac{\|k_{a}\|_r}{\|k_{a}^I\|_r}
   \frac{k_{a}}{\|k_{a}\|_r}\right\|_r^r\nn\\
 &\simeq& \sum_{a\in S}|\mu_a|^r
 \left(\frac{\|k_{a}\|_q}{\|k_{a}^I\|_r}\right)^r
 \simeq \sum_{a\in S}|\mu_a|^r,
\eea
where we have used that
$\|k_{a}\|_r\simeq\|k_{a}^I\|_r$. This holds in particular for $r=q$, where
$1/s=1/p+1/q$.
\end{proof}

%
We are now in a position to deduce also the second part of 
Theorem \ref{thm1}.

\begin{corollary}\label{cor4.8}
Let $1<p\le 2$.
Let $I$ be a one-component singular inner function and
$S\subset \DD$. Suppose
that $\sup_{a\in S}|I(a)|<1$. If
$(k_{a}^I/\|k_{a}^I\|_{p'})_{a\in S}$
is uniformly minimal in $K^{p'}_I$, where $1/p+1/p'=1$ 
then for every $\eps>0$ and
for every $q<p$, 
$(k_{a}^I/\|k_{a}^I\|_{q'})_{a\in S}$
is an unconditional basis in $K^{q'}_{I^{1+\eps}}$.
\end{corollary}

So in the present situation, we {\it increase} the space in the
direction of the inner function and we {\it decrease} the space
by increasing the power of integration to deduce unconditionality
from uniform minimality.

Let us make another observation.
In \cite[D4.4.9(5)]{Nik} it is stated
(in conjunction with \cite[Lemma D4.4.3]{Nik}) that
under the Carleson condition $S\in (C)$
the condition $\sup_{a\in S}|I(a)|<1$ is equivalent
to the existence 
of $N\in\N$ such
that $(k_a^{I^N}/\|k_a^{I^N}\|_2)_{a\in S}$ is 
an unconditional sequence in $K^2_{I^N}$.
In the present situation, when $(k_a^{I}/\|k_a^{I}\|_{p'})
_{a\in S}$, $p'\ge 2$,
is supposed uniformly minimal (which itself implies the Carleson
condition under the assumptions on $I$ and $S$; we do not know
whether the Carleson condition could imply the uniform minimality
in our context) then instead
of taking $I^N$ we can choose $I^{1+\eps}$ for any $\eps>0$ (paying
the price of replacing $p'$ by $q'>p'$).

\begin{proof}[Proof of Corollary \ref{cor4.8}]
In view of the preceding corollary and Corollary \ref{cor2.3},
it remains to check that $S$ is $(l^s)^*=l^{s'}$-Carleson, which
follows at once from \eqref{scarl} by taking $r=s'$.
\end{proof}


\begin{proof}[Proof of the theorem]
It remains to prove that the hypotheses of the theorem imply those of 
Lemma~\ref{algUniExt411}.
We thus have to prove that
\beqa
 \displaystyle {\mathbb{E}}\displaystyle 
  \left[{\displaystyle \displaystyle \left\Vert{\displaystyle 
  \displaystyle \sum_{a\in S}^{}{\lambda _{a}\epsilon _{a}\rho
    _{p,a}}}
   \right\Vert _{p}^{p}}\right] 
 \lesssim \displaystyle \left\Vert{\displaystyle \lambda }
   \right\Vert _{\displaystyle \ell ^{p}}^{p}.
\eeqa
under the assumption that the dual sequence 
$\{\rho _{p,a}\}_{a\in S}$ is uniformly bounded in 
$\displaystyle K_{I}^{p}$: 
$\displaystyle \sup _{a\in S}\ \displaystyle \Vert{\displaystyle 
\rho _{p,a}}\Vert _{p}\leq C.$

By Fubini's theorem
\beqa
 \displaystyle {\mathbb{E}}\displaystyle 
  \left[{\displaystyle \displaystyle \left\Vert{\displaystyle 
 \displaystyle \sum_{a\in S}^{}{\lambda _{a}\epsilon _{a}\rho _{p,a}}}
  \right\Vert _{p}^{p}}\right] 
 =\displaystyle \int_{{\mathbb{T}}}^{}{{\mathbb{E}}\displaystyle 
 \left[{\displaystyle \displaystyle \left\vert{\displaystyle \displaystyle 
  \sum_{a\in S}^{}{\lambda _{a}\epsilon _{a}\rho _{p,a}}}\right\vert^{p}}
 \right] \,d\sigma },
\eeqa
and by Khinchin's inequalities we have
\beqa
 \displaystyle {\mathbb{E}}\displaystyle \left[{\displaystyle 
  \displaystyle \left\vert{\displaystyle 
  \displaystyle \sum_{a\in S}^{}{\lambda _{a}\epsilon _{a}\rho
    _{p,a}}}\right\vert ^{p}}\right] 
 \simeq \displaystyle \left({\displaystyle \displaystyle \sum_{a\in
       S}^{}{\displaystyle \left\vert{\displaystyle \lambda _{a}}
  \right\vert^{2}\displaystyle \left\vert{\displaystyle \rho _{p,a}}
  \right\vert ^{2}}}\right) ^{p/2}.
\eeqa
Now, since $p\leq 2$, we have
\beqa
 \displaystyle \displaystyle \left({\displaystyle \displaystyle 
  \sum_{a\in S}^{}{\displaystyle \left\vert{\displaystyle \lambda _{a}}
  \right\vert ^{2}\displaystyle \left\vert{\displaystyle 
  \rho _{p,a}}\right\vert ^{2}}}\right) ^{1/2}
 \leq \displaystyle \left({\displaystyle 
  \displaystyle \sum_{a\in S}^{}{\displaystyle 
  \left\vert{\displaystyle \lambda _{a}}\right\vert ^{p}\displaystyle 
  \left\vert{\displaystyle \rho _{p,a}}\right\vert ^{p}}}\right)^{1/p},
\eeqa
and hence
\beqa
 \displaystyle \displaystyle \int_{{\mathbb{T}}}^{}{{\mathbb{E}}
  \displaystyle \left[{\displaystyle \displaystyle 
  \left\vert{\displaystyle \displaystyle \sum_{a\in S}^{}{\lambda _{a}
  \epsilon _{a}\rho _{p,a}}}\right\vert ^{p}}\right] \,d\sigma }
 \leq \displaystyle \int_{{\mathbb{T}}}^{}{\displaystyle 
  \left({\displaystyle \displaystyle \sum_{a\in S}^{}{\displaystyle 
  \left\vert{\displaystyle \lambda _{a}}\right\vert ^{p}\displaystyle 
  \left\vert{\displaystyle \rho _{p,a}}\right\vert ^{p}}}\right)}\,d\sigma 
 =\displaystyle \sum_{a\in S}^{}{\displaystyle 
  \left\vert{\displaystyle \lambda _{a}}\right\vert ^{p}\displaystyle 
  \left\Vert{\displaystyle \rho _{p,a}}\right\Vert _{p}^{p}}.
\eeqa
So, finally
\beqa
 \displaystyle {\mathbb{E}}\displaystyle 
  \left[{\displaystyle \displaystyle \left\Vert{\displaystyle 
   \displaystyle \sum_{a\in S}^{}{\lambda _{a}\epsilon _{a}\rho
     _{p,a}}}\right\Vert _{p}^{p}}\right] 
 \lesssim \sup _{a\in S}\ \displaystyle \left\Vert{\displaystyle 
    \rho _{p,a}}\right\Vert _{p}^{p}\displaystyle \left\Vert{\displaystyle 
    \lambda }\right\Vert _{p}^{p},
\eeqa
and consequently the theorem holds. 
\end{proof}

\end{document}